\documentclass[11pt]{amsart}

\usepackage{amsmath, amssymb, amsthm, epsfig}
\usepackage{hyperref, latexsym,a4wide, tikz}
\usepackage{url}
\usepackage[mathscr]{euscript}
\usepackage{color}
\usepackage{url}

\usepackage{setspace}

\def\today{\ifcase\month\or
  January\or February\or March\or April\or May\or June\or
  July\or August\or September\or October\or November\or December\fi
  \space\number\day, \number\year}
\DeclareMathOperator{\sgn}{\mathrm{sgn}}

\newtheorem{theorem}{Theorem}

\newtheorem{lemma}{Lemma}
\newtheorem{proposition}[theorem]{Proposition}

\theoremstyle{definition}

\theoremstyle{remark}

\newcommand{\T}{\mathbb{T}}

\newcommand{\C}{\mathbb{C}}

\newcommand{\R}{\mathbb{R}}

\newcommand{\Q}{\mathbb{Q}}
\newcommand{\Z}{\mathbb{Z}}

\newcommand{\ba}{{\boldsymbol{a}}}
\newcommand{\bu}{{\boldsymbol{u}}}

\newcommand{\bs}{{\boldsymbol{s}}}

\newcommand{\dt}{\text{\rm d}t}

\renewcommand{\d}{\text{\rm d}}

\newcommand{\la}{\lambda}

\newcommand{\f}[1]{\lfloor #1 \rfloor}

\begin{document}

%------------------HEADINGS------------------------

\title[]{A Universality law for sign correlations\\ of eigenfunctions of differential operators}

\author[]{Felipe Gon\c{c}alves}
\address{
        Hausdorff Center for Mathematics,
        53115 Bonn, Germany}
\email{goncalve@math.uni-bonn.de}

\author[]{Diogo Oliveira e Silva}
\address{
     School of Mathematics, University of Birmingham, Edgbaston, Birmingham,
B15 2TT, England}
\email{d.oliveiraesilva@bham.ac.uk}

\author[]{Stefan Steinerberger}
\address{
Department of Mathematics\\
        Yale University\\
        New Haven, CT 06511, USA}
\email{stefan.steinerberger@yale.edu}

\thanks{D. O. S. acknowledges partial support by the College Early Career Travel Fund of the University of Birmingham.
 S.S. is supported by the NSF (DMS-1763179) and the Alfred P. Sloan Foundation.}


\allowdisplaybreaks
%\numberwithin{theorem}{subsection}{section}
%\numberwithin{equation}{section}

%------------------ABSTRACT------------------------------

\begin{abstract} We establish a universality law for sequences of functions $\{w_n\}_{n \in \mathbb{N}}$ satisfying a form of WKB approximation on compact intervals. This includes eigenfunctions of generic Schr\"odinger operators, as well as Laguerre and Chebyshev polynomials. Given two distinct points $x, y \in \mathbb{R}$, we ask how often do $w_n(x)$ and $w_n(y)$ have the same sign. Asymptotically, one would expect this to be true half the time, but this turns out to not always be the case.
Under certain natural assumptions, we prove that, for all $x \neq y$,
$$  \frac{1}{3} \leq  \lim_{N \rightarrow \infty}{ \frac{1}{N} \#  \left\{0 \leq n < N:  \sgn(w_n(x)) = \sgn(w_n(y)) \right\}} \leq \frac{2}{3},$$
and that these bounds are optimal, and can be attained.
%and there exist extremizing pairs $(x^\star,y^\star)\in\Q^2$. 
Our methods extend to other questions of similar flavor and we also discuss a number of open problems.
\end{abstract}

%---------------------TITLE--------------------------------

\maketitle

%---------------------HAVE--FUN!-----------------------------------

\section{Introduction}

\subsection{Setup}
This paper is concerned with a simple and surprising property exhibited by the sequence of eigenfunctions 
for the eigenvalue problem of certain differential operators. Consider, on the real line, the Schr\"odinger operator associated to the potential $V$, 
\begin{equation}\label{eq:SchrOp}
H=- \frac{\d^2}{\d x^2}+V(x).
\end{equation}
Here, $V:\R\to\R$ is some increasing function satisfying $V(x) \rightarrow \infty$, as $|x| \rightarrow \infty$.

\begin{center}
\begin{figure}[h!]\begin{tikzpicture}[scale=2.2]
      \draw[->] (-2.2,0) -- (2.2,0) node[right] {$x$};
      \draw[->] (0,0) -- (0,1.2) node[above] {$y$};
      \draw[scale=0.5,domain=-3:3,smooth,variable=\x, thick, dashed] plot ({\x},{0.2*\x*\x});
      \draw[scale=0.5,domain=-3:3,smooth,variable=\x,ultra thick] plot ({\x},{exp(-\x*\x)});
      \draw[scale=0.5,domain=-3:3,smooth,variable=\x,ultra thick] plot ({\x},{2*\x*exp(-\x*\x)});
      \draw[scale=0.5,domain=-3:3,smooth,variable=\x,ultra thick] plot ({\x},{(4*\x*\x - 2)*exp(-\x*\x)});
    \end{tikzpicture}
\vspace{-5pt}
\caption{ The potential $V(x) = x^2$ (dashed) and the first three eigenfunctions of the quantum harmonic oscillator.}
\label{fig:QuadraticPotential}
\end{figure}
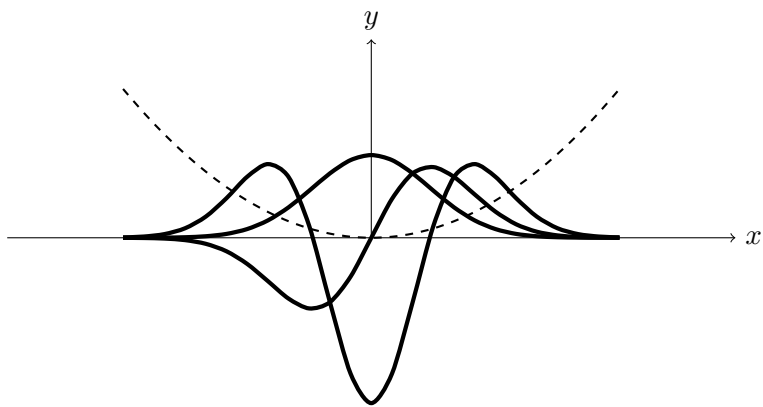
\end{center}

The eigenvalue problem
\begin{equation}\label{eq:EgvProb}
H(w_n)=\lambda_n w_n
\end{equation}
has been studied extensively, the simpler case $V(x) = x^2$ corresponding to the quantum harmonic oscillator
whose eigenfunctions are given by the Hermite functions (see Figure \ref{fig:QuadraticPotential}). It is well understood that, as the eigenvalues become large, the second derivative dominates, and the eigenfunctions start to look locally like trigonometric functions.
This phenomenon is known as {\it WKB approximation}, named after Wentzel, Kramers, and Brillouin. 

The purpose of our paper is to establish a rather surprising universality statement for sign correlations of sequences of functions for which a kind of WKB approximation holds.
Our starting point is  very simple to state: 
 given two distinct points $x, y \in \mathbb{R}$, how often do $w_n(x)$ and
$w_n(y)$ have the same sign? More precisely, we are interested in the {\it sign correlation limit}, defined as
%\footnote{This cannot be replaced by $w_n(x)w_n(y)>0$, right?}
\begin{equation}\label{eq:SCL} 
\ell_{\{w_n\}}(x,y):=\lim_{N \rightarrow \infty}{ \frac{1}{N} \#  \left\{0 \leq n < N:   \sgn(w_n(x))=\sgn(w_n(y)) \right\}}.
\end{equation}
One could be tempted to conjecture that, in the high frequency limit, the two points $x,y$ decouple and the corresponding signs behave essentially like independent Bernoulli random variables, thus exhibiting the same sign in roughly half of the cases. 
This seemingly natural conjecture turns out to be a good guess
for the generic behavior of the system. 
However, earlier work of the authors \cite{us3} hinted at the possible existence of an exceptional set exhibiting a different kind of behavior, and motivated the present paper.

\subsection{Main Result.} 
A sequence $\{a_n\}\subseteq [0,1]$ is said to be {\it equidistributed in $[0,1]$} if, for any subinterval $[c,d]\subseteq [0,1]$,
$$\lim_{N\to\infty} \frac1N \# \{0\leq n<N:  a_n\in[c,d] \}={d-c}.$$
A sequence $\{a_n\}\subseteq \R$ is said to be {\it equidistributed modulo 1} if the sequence of the fractional parts $\{a_n -\f{a_n}\}$ is equidistributed in  $[0,1]$.
Our first main result applies to a sequence of functions obeying a certain asymptotic behavior
which is inspired by the WKB approximation, and is satisfied by several classical objects 
(see the examples in \S \ref{examples}). 
Regarding notation, $o_n(1)$ will denote a quantity that tends to 0, as  $n\to\infty$. We will also write 
%$a_n\sim b_n$ if $\lim_{n\to\infty} a_n/b_n=1$, and 
$a_n=O(b_n)$, or $|a_n|\lesssim |b_n|$, if there exists a constant $C<\infty$ (independent of $n$) such that $|a_n|\leq C |b_n|$, for every $n$.

\begin{theorem}[Main Result] \label{mainthm}
 Given $D \subseteq \mathbb{R}$, let $w_n: D \rightarrow \mathbb{R}$ be a sequence of functions satisfying
 %\footnote{Is there a reason not to replace $\phi_1(x)\phi_2(n)$ in the previous version by a simpler, more general and nicer-looking $\phi(x,n)$?} 
 % on every interval $I \subset D$,
\begin{equation}\label{eq:H1asympt}
 w_n(x) = \left(1 + o_n(1) \right) \phi(x,n) \cos{\left( 2\pi (\mu_n \varphi(x) - \theta) \right)},
 \end{equation}
for every $x \in D$ and some $\{\mu_n\} \subset \mathbb{R}$, $\theta\in\R$, and function $\phi:D\times\mathbb{N}\to \R$.
Consider distinct points $x, y \in D$ such that
%\footnote{In order to ensure that the correlation limit is not 1, we have to make some further assumptions on $x,y$. Please double check that this does the job as intended.} 
$\varphi(x) \neq \pm \varphi (y)$ and $\phi(x,n)\phi(y,n)>0$ for all $n$.
If the sequences $\{p^{-1}\mu_n\varphi(x)\}$ and $\{q^{-1}\mu_n\varphi(y)\}$ are equidistributed modulo 1 for any $p,q\in \Z\setminus \{0\}$, then the sign correlation limit \eqref{eq:SCL}  exists, and satisfies
\begin{equation}\label{eq:upperlower}
  \frac{1}{3} \leq  \ell_{\{w_n\}}(x,y) \leq \frac{2}{3}.
  \end{equation}
Moreover, these constants are optimal.
% and there exist extremizing pairs $(x_{\emph{min}}^\star,y_{\emph{min}}^\star)\in\Q^2$ and $(x_{\emph{max}}^\star,y_{\emph{max}}^\star)\in\Q^2$ such that 
%$$\ell_{\{w_n\}}(x_{\emph{min}}^\star,y_{\emph{min}}^\star)=\frac13 \text{ and }\ell_{\{w_n\}}(x_{\emph{max}}^\star,y_{\emph{max}}^\star)=\frac23.$$ 
\end{theorem}

We believe this result to be rather surprising.
In particular, it establishes the existence of  correlations different from $\frac12$. These correlations are, however, universally bounded away from both 0 and 1. Theorem \ref{mainthm} motivates a number of natural questions, see \S 1.4 below.

 %We do not expect a higher dimensional version of Theorem \ref{mainthm} to hold at a similar level of generality, since the phenomenon seems to be genuinely one-dimensional, just like the WKB approximation.

\subsection{Sharper asymptotics.} The sign correlation limit can be computed {\it exactly} in a number of situations of interest.
We proceed to describe one such situation.
 Let $V\in L^1_{\text{loc}}(\R)$ be a locally integrable potential such that $V(x)\to \infty$, as $|x|\to \infty$, and assume $V$ to be bounded from below, 
 \begin{equation}\label{eq:essinf}
 \text{ess inf}_{x\in\R} V(x)>-\infty.
 \end{equation} 
We renormalize the Hamiltonian by setting $H_V=-\frac{1}{4\pi^2}\frac{d^2}{dx^2}+V(x)$ (this is adapted to our choice of normalization for the Fourier transform, see \eqref{eq:FourierNormalization} below). Under these conditions,  the operator $H_V$ given by \eqref{eq:SchrOp} is known to have compact resolvent.
 In particular, $H_V$ has purely discrete spectrum and a complete set of eigenfunctions, see \cite[Theorem XIII.67]{RS72}. 
 This means that there exists an orthogonal basis $\{w_n\}$ of $L^2(\R)$ such that  $H_V (w_n)=\lambda_n w_n$, where the eigenvalues $\{\lambda_n\}$ form a non-decreasing sequence satisfying $\lambda_n\to \infty$, as $n\to \infty$. 
 In addition, we require $V$ to be an even function. This implies that the basis $\{w_n\}$ naturally splits into even and odd functions, since these subspaces are $H_V$-invariant. In particular, we can reorder the basis elements in such a way that $w_n$ is an even function if $n$ is even, and an odd function if $n$ is odd. After doing so, the sequence $\{\lambda_n\}$ may no longer be non-decreasing, however we still have that $\lambda_n\to\infty$, as $n\to\infty$. 
 By uniqueness of solutions to the eigenvalue problem \eqref{eq:EgvProb}, we may further impose
  $\text{sgn}(w_{2n}(0)) = \text{sgn}(w'_{2n+1}(0)) =(-1)^n$.
Here and in the rest of the paper a prime denotes differentiation with respect to the variable $x$.
 We will also require both subsequences $\{\sqrt{\lambda_{2n} }x\}$ and $\{\sqrt{\lambda_{2n+1}}x\}$ to be equidistributed modulo 1, for every $x\neq 0$. Whether this should  generically be the case is discussed in Problem (3) from \S 1.4 below.
We are now ready to state our second main result.

\begin{theorem}[Sharper asymptotics]\label{mainthm2}
Let $V\in L_{\emph{loc}}^1(\R)$ be an even potential, bounded from below in the sense of \eqref{eq:essinf}, and such that $V(x)\to\infty$, as $|x|\to\infty$. For each $n\in\mathbb{N}$, assume that for the associated eigenvalue problem $H_V(w_n)=\lambda_n w_n$, the following assertions hold:
\begin{itemize}
\item[(H1)] the function $w_n$ is  even  if $n$ is even, and  odd  if $n$ is odd,
\item[(H2)] the sequences $\{\sqrt{\lambda_{2n}}x\}$ and $\{\sqrt{\lambda_{2n+1}}x\}$ are equidistributed modulo 1, for any $x\in \R\setminus \{0\}$,
\item[(H3)] and we have $\sgn(w_{2n}(0)) = \sgn(w'_{2n+1}(0)) =(-1)^n$.
\end{itemize}
%For each $n\geq 0$, set $c_n := (w_n(0)^2+\la_n^{-1}w_n'(0)^2)^{1/2}.$
Then the  asymptotic
\begin{equation}\label{eq:asymptotic}
w_{n}(x) = (1+o_n(1))\left(w_n(0)^2+\tfrac{w_n'(0)^2}{4\pi^2\la_n}\right)^{1/2}
\cos\left(2\pi\left(\sqrt{\lambda_{n}}x- \tfrac {n}4\right)\right)
\end{equation}
holds uniformly on compact subsets of the real line. 
 If $x,y$ are distinct real numbers such that
 $\frac xy=\frac pq$ for some nonzero coprime integers $p,q$, then the sign correlation limit \eqref{eq:SCL} is given by
\begin{equation}\label{eq:SCLpotentials}
\ell_{\{w_n\}}(x,y) 
 = \left\{
\begin{array}{lc}
 \frac12+\frac{1}{2pq} & \text{if } p \equiv q\equiv 1\mod 4, \text{ or }p \equiv q\equiv  3 \mod 4, \\
\frac12 & \text{otherwise}.
\end{array}
\right.
\end{equation}
If $\frac xy$ is irrational, then $\ell_{\{w_n\}}(x,y) = \frac{1}{2}.$
\end{theorem}

The asymptotic \eqref{eq:asymptotic} is exactly the one given via WKB approximation. 
 The quadratic case $V(x)=x^2$, where the WKB approximation coincides with the classical asymptotic for Hermite polynomials, falls under the scope of Theorem \ref{mainthm2} and is described in more detail in \S \ref{examples} (together with
 higher dimensional extensions, provided by the Laguerre polynomials).

\subsection{Further remarks and open problems.}

\begin{enumerate}

\item Can Theorems \ref{mainthm} and \ref{mainthm2} be extended to sign correlations of three or more points? 
What can be said about the density with which a specific sign configuration, say $(+, -, +,-,-)$, can occur? 
Some of these may be universally bounded away from 0 and 1, while others may not be.
In principle, our approach provides a framework for obtaining such bounds since each such question is reduced to a finite  computation. However, the increase in complexity is substantial, which is why we have not been able to further explore this question. We believe it to be a promising avenue for future research.\\
\item Is it possible to characterize the class of potentials $V$ such that our result applies to eigenfunctions of the Schr\"odinger operator $H_V$? The WKB approximation seems to be a valuable tool; however, it is not clear to us whether a suitable theory on the equidistribution of the eigenvalues of differential operators exists.
On the other hand, the asymptotic growth of $\{\lambda_n\}$, as $n\to\infty$,
 has been studied extensively, see e.g. \cite{fefferman}, 
 %and they suggest that they should be equidistributed in most cases; 
 but this question seems more subtle.\\

\item  %In
%\footnote{This needs to be rewritten. There is no question here, and as it stands it's just repeating what \S 4 does. Maybe we leave for after \S 4 is polished?} 
%\S 4 below we hint at possible connections between these questions and classical problems regarding the asymptotic behavior of geodesics on the torus $\mathbb{T}^2$. 
As we shall see, these questions are connected to classical problems on the asymptotic behavior of geodesics on the $d$-dimensional torus $\T^d$.
It is natural to expect that several of the new developments regarding strong forms of linear flow rigidity \cite{beck0, beckh, beck, beck2, grep, stein} can be used to make more precise statements in some special cases.
We also note that at least for some classical families of orthogonal polynomials it should be possible to
obtain more precise quantitative information -- see \S 4 below for further details.
\end{enumerate}
\smallskip

\section{Useful Lemmata}%\label{sec:lemmata}

We start with a general result that will serve as a first step towards computing the sign correlation limit of a sequence of functions over a fixed finite set of points $\ba=(a_1,...,a_d)\in \R^d$. 

\begin{lemma}\label{limit-lemma}
Given 
%\footnote{F: As it stands, this lemma involves 8 variables $(d,a,s,\lambda,p,\mu,\alpha,u)$ and is hard to parse. Now that we know exactly what it is useful for, can we write it in a more transparent way?} 
$\ba\in \R^d$, assume that $\lambda \ba\in \Z^d$, for some $\lambda>0$. Let $f:\R\to\R$ be a continuous $1$-periodic 
%\footnote{F: The Riemann integrability hypothesis looks a bit odd to me. We will use this lemma for $f(z)=\cos (2\pi z)$, or translations thereof, so what about just "continuous"?} 
function. Let $\{\mu_n\}\subset\R$ be a sequence such that $\{\frac{\mu_n}{\lambda}\}$ is equidistributed modulo $1$. 
%\footnote{F: We are just using the fact that the sequence $\{\mu_n/\lambda\}$ is equidistributed mod 1 (note that $\lambda$ has already been fixed in the statement of the lemma), and we when we apply this to prove the main theorems we won't have equidistribution for all $\alpha$. So we should change this bit, no?}. 
Let $\bs\in\{-1,1\}^d$.
Then 
\begin{align*}%\label{limit-intergral}
\lim_{N\to\infty} \frac 1N\#\{0\leq n< N: (\sgn[f(\mu_n a_1)],...,\sgn[f(\mu_n a_d)])=\pm \bs\} = \int_0^1 \Psi(\lambda t \ba)\,\dt,
\end{align*}
where the function $\Psi:\R^d\to\{0,1\}$ is defined as follows: given $\bu\in\R^d$, then $\Psi(\bu)=1$ 
if $(\sgn[f(u_1)],...,\sgn[f(u_d)])=\pm\bs$, and $\Psi(\bu)=0$ otherwise.
\end{lemma}

\begin{proof}
Consider the function $g(t):=\Psi(\lambda t \ba)$, which 
%takes values in $\{0,1\}$ and 
satisfies $g(t+1)=g(t)$, for every $t\in\R$. 
By construction, we have that
$$\left\{0\leq n< N: (\sgn[f(\mu_n a_1)],...,\sgn[f(\mu_n a_d)])=\pm \bs\right\} = \{0\leq n< N: g(\tfrac{\mu_n}{\lambda})=1\}.$$ 
Since the function $g$ is 1-periodic and the sequence $\{\frac{\mu_n}{\lambda}\}$ is equidistributed modulo 1, we have that, as $N\to\infty$,
$$\frac1N \#\{0\leq n< N: (\sgn[f(\mu_n a_1)],...,\sgn[f(\mu_n a_d)])=\pm \bs\}  \to |\{t\in [0,1]: g(t)=1\}| = \int_{[0,1]} g.$$
The last identity follows from the fact that the function $g$ takes values in $\{0,1\}$.
This concludes the proof of the lemma.
\end{proof}

Only the case $d=2$ of Lemma \ref{limit-lemma} will be relevant to our applications.
For the remainder of the section, we will discuss integrals of the function
%$$\Psi(x,y):=\text{sgn}(\sin(2\pi x)\sin(2\pi y))$$ and
\begin{equation}\label{eq:DefPhi}
\Phi(x,y):=\text{sgn}(\cos(2\pi x)\cos(2\pi y))
\end{equation}
over rays of the two-dimensional torus $\mathbb{T}^2=\R^2/\Z^2$,
which will play a key role in the proof of our main theorems.
We remark that the Haar measure on $\mathbb{T}^2$ coincides with the Lebesgue measure on the fundamental domain $[0,1]^2$.
%Note that 
%$\Phi(x+1/4,y+1/4)=\Psi(x,y)$ and that 
We further note that, given a ray $\gamma:\R\to\mathbb{T}^2$ defined by $\gamma(t)=(At-\alpha,Bt-\beta)$ for some $A,B\neq 0$, then
\begin{equation}\label{eq:cov}
 \lim_{T \rightarrow \infty}{ \frac{1}{T} \int_{0}^{T}{ \Phi(\gamma(t))  \,\d t}} =  \lim_{T \rightarrow \infty}{ \frac{1}{T} \int_{0}^{T}{ \Phi(\widetilde \gamma(t))  \,\d t}},
\end{equation}
where $\widetilde \gamma:\R\to\mathbb{T}^2$ is in turn given by $\widetilde \gamma(t)=(t,at+b)$, with $a=B/A$ and $b=(B/A) \alpha-\beta$.
 The following lemma is well-known, with suitable modifications and vast generalizations appearing in 
 %\footnote{S: More precise reference withtin the books? Page, section, theorem,...?} 
 \cite{dick, drmota, kuipers}.
 For the sake of completeness, we provide a short proof.
\begin{lemma}\label{lem:Ergodic}
Given $a,b\in\R$, let $\gamma(t) = (t, at + b)$ be the corresponding line in $\mathbb{R}^2$. Then the limit
\begin{equation}\label{eq:TimeSpaceAvg}
 \lim_{T \rightarrow \infty}{ \frac{1}{T} \int_{0}^{T}{ \Phi(\gamma(t))  \,\d t}}
 \end{equation}
exists. Moreover, if the limit is nonzero, then the coefficient $a$ is a rational number.
\end{lemma}
\begin{proof} Since the function $\Phi$ is 1-periodic in the variables $x$ and $y$, the problem reduces to a standard question in equidistribution theory on the 2-dimensional torus $\T^2$. 
If  $a$ is irrational, then the line $t\mapsto (t,at+b)$ is densely wound and equidistributes over $\T^2$, and the averaged integral in \eqref{eq:TimeSpaceAvg} converges to the average value of $\Phi$,
$$ \lim_{T \rightarrow \infty}{ \frac{1}{T} \int_{0}^{T}{ \Phi(\gamma(t)) \,\d t}} =  \int_{\mathbb{T}^2}{\Phi} = 0,$$
see \cite[\S 2.3]{drmota}.
If $a$ is rational, then the line $t\mapsto (t,at+b)$ gives rise to a closed geodesic on $\T^2$, and the existence of the limit \eqref{eq:TimeSpaceAvg} follows from periodicity.
% (but, of course, does not have to be 0).
\end{proof}
The next  lemma further analyzes the case of rational slope $a= p/q\in\Q$. 
It is  of quantitative flavor, and relies on the explicit form of the function $\Phi$. 
We achieve this by resorting to Fourier series, and will
 normalize the Fourier coefficients of an integrable function $f:[0,1]\to\C$ in the following way: 
\begin{equation}\label{eq:FourierNormalization}
\widehat{f}(n)=\int_0^1 f(x) e^{-2\pi i n x}\,\d x.
\end{equation}

\begin{lemma}\label{lem:Fourier}
Let $A,B\in \R$ be nonzero real numbers, such that $A/B = p/q$ for some coprime $p,q\in\Z$. Let $\alpha,\beta\in \R$ and let $\gamma(t) = (At-\alpha, Bt-\beta)$ be the corresponding ray on $\mathbb{T}^2$. 
If either $p$ or $q$ are even, then
$$
 \lim_{T \rightarrow \infty}  \frac{1}{T} \int_{0}^{T} \Phi(\gamma(t)) \,\d t  = 0.
$$
If both $p$ and $q$ are odd, then
$$
 \lim_{T \rightarrow \infty}  \frac{1}{T} \int_{0}^{T} \Phi(\gamma(t)) \,\d t = (-1)^{\frac{p+q}2+1}\frac{8}{\pi^2 pq} \sum_{\ell=0}^\infty \frac{\cos\left(2\pi (2\ell+1) (p\beta-q\alpha)\right)}{(2\ell+1)^2}.
$$

In particular, in this case, we have that
\begin{equation*}%\label{eq:goalineq}
\left| \lim_{T \rightarrow \infty}  \frac{1}{T} \int_{0}^{T}{ \Phi(\gamma(t))\, \d t} \right| \leq  \frac 1{|pq|}, \,\,\,(p,q \text{ odd})
\end{equation*}
where equality is attained if and only if $p\beta-q\alpha$ is an integer.
\end{lemma}

\begin{proof}
By periodicity, recall \eqref{eq:cov}, we have that
\begin{equation*}%\label{eq:periodic}
\lim_{T \rightarrow \infty}  \frac{1}{T} \int_{0}^{T}{ \Phi(\gamma(t))\, \d t}
=\int_0^1\Phi(pt-\alpha,qt-\beta)\,\d t.
\end{equation*}
Expanding the function $\Phi$ in Fourier series,
$$\Phi(x,y) = \frac{4}{\pi^2}\sum_{\substack{n,m\in\Z \\ m,n\neq 0}} \frac{\sin(\frac{\pi n}2)\sin(\frac{\pi m}2)}{mn} e^{2\pi i(mx+ny)},$$
we  obtain that
\begin{align*}
\int_0^1 \Phi(pt-\alpha,qt-\beta)\,\d t & =  \frac{4}{\pi^2}\sum_{\substack{n,m\in\Z \\ m,n\neq 0}} 
\frac{\sin(\frac{\pi n}2)\sin(\frac{\pi m}2)}{mn} \int_0^1 e^{2\pi i(mp+nq)t} e^{-2\pi i(m\alpha+n\beta)}\, \d t \\ 
& = \frac{8}{\pi^2pq} \sum_{k=1}^\infty \frac{\sin(\frac{\pi kp}2)\sin(\frac{\pi kq}2)}{k^2} \cos(2\pi k(p\beta-q\alpha)).
\end{align*}
This quantity vanishes if either $p$ or $q$ are even. 
On the other hand, if both $p$ and $q$ are odd, then 
$$
\int_0^1 \Phi(pt-\alpha,qt-\beta)\,\d t 
= (-1)^{\frac{p+q}2+1}\frac{8}{\pi^2 pq} \sum_{\ell= 0}^\infty \frac{\cos(2\pi (2\ell+1) (p\beta-q\alpha))}{(2\ell+1)^2}.
$$
Since
$\sum_{\ell= 0}^{\infty} \frac{1}{(2\ell+1)^2} = \frac{\pi^2}{8},$
the triangle inequality implies
$$
\left|\int_0^1 \Phi(pt-\alpha,qt-\beta)\,\d t\right| \leq \frac{1}{|pq|},
$$
where equality is attained if and only if $(p\beta-q\alpha)\in \Z$.
\end{proof}

\section{Proofs of Theorems \ref{mainthm} and \ref{mainthm2}}

\begin{proof}[Proof of Theorem \ref{mainthm}] 
Let $x\neq y \in D$ be given, satisfying $\varphi(x)\neq\pm\varphi (y)$ and  $\phi(x,n)\phi(y,n)>0$, for all $n$. 
%Recall \eqref{eq:H1asympt}.
No generality is lost in assuming that $\varphi(x)/\varphi(y)=p/q$ for some coprime $p,q\in\Z$, 
and that $p/\varphi(x)= q/\varphi(y)>0$, for otherwise Lemma \ref{lem:Ergodic} would imply that $\ell_{\{w_n\}}(x,y)=\frac12$, and there is nothing to prove.
Now, since $\phi(x,n)\phi(y,n)>0$ for all $n$, the asymptotic \eqref{eq:H1asympt} implies 
%the following sign correlation limits coincide:
\begin{equation}\label{eq:equallim}
\ell_{\{w_n\}}(x,y)=\ell_{\{u_n\}}(x,y),
\end{equation}
where $u_n(x):=\cos(2\pi(\mu_n\varphi(x)-\theta))$.
% $\sgn(w_n(x))=\sgn(w_n(y))$ if and only if $\sgn(\cos(\mu_n \phi_3(x)-\theta))=\sgn(\cos(\mu_n \phi_3(y)-\theta))$.
We focus on the latter limit, and prepare to apply
%$$\ell_{\{\cos(\mu_n\phi_3(\cdot)-\theta)\}}(x,y)=\lim_{N\to\infty}\frac 1N\#\{0\leq n<N: \sgn(\cos(\mu_n\phi_3(x)-\theta))=\sgn(\cos(\mu_n\phi_3(y)-\theta))\}.$$
 %Let  $H(x):=\frac \d{\d x}\max\{0,x\}$ denote the Heaviside step function
  Lemma \ref{limit-lemma} with $\ba=(\varphi(x),\varphi(y))$, $\bs=(1,1)$, $\lambda=p/\varphi(x)=q/\varphi(y)$, and $f(z)=\cos(2\pi(z-\theta))$.
  Note that the our equidistribution assumption implies that the sequence $\{\tfrac{\mu_n}{\lambda}\}=\{p^{-1}\mu_n\varphi(x)\}$ is equidistributed modulo 1, and so all the hypotheses of Lemma \ref{limit-lemma} are satisfied.
The conclusion is that  
  \begin{equation}\label{eq:secondlim}
  \ell_{\{u_n\}}(x,y)=\int_0^1 \Psi(pt-\theta,qt-\theta)\,\d t,
  \end{equation}
  where the function $\Psi$ is related to $\Phi$ from \eqref{eq:DefPhi} via $\Phi=2\Psi-1$.
It then follows from \eqref{eq:equallim} and \eqref{eq:secondlim} that
  \begin{equation}\label{eq:TL}
  \ell_{\{w_n\}}(x,y)=\frac12+\frac12\int_0^1 \Phi(pt-\theta, qt-\theta) \,\d t.
  \end{equation}
The latter integral was computed in the course of the proof of Lemma \ref{lem:Fourier}, and is non-zero only if both $p,q$ are odd. In that case, applying Lemma \ref{lem:Fourier} with $A=\varphi(x)$, $B=\varphi(y)$, and $\alpha=\beta=\theta$, yields
  \begin{equation}\label{eq:fourthlim}
\left|\int_0^1 \Phi(pt-\theta, qt-\theta)\, \d t\right|\leq \frac1{|pq|}.
  \end{equation}%\label{eq:thirdlim}
To finish the argument, note that $p,q$ both being odd, and $\varphi(x)\neq\pm\varphi(y)$, jointly force the inequality $\frac1{|pq|}\leq \frac13$. Estimates \eqref{eq:TL} and \eqref{eq:fourthlim} then imply \eqref{eq:upperlower}, which is the first desired conclusion.
To verify the claimed optimality, %\footnote{F,S: Can we characterize the cases of equality in \eqref{eq:upperlower}?  Stefan: sure, didn't we have that at some point? the optimal slopes are only 1/3 and 3?} 
recall the cases of equality in Lemma \ref{lem:Fourier} and consider the particular case when $\varphi(x)=3\varphi(y)$ and $\theta(\varphi(x)-\varphi(y))\in\Z$.
\end{proof}

\begin{proof}[Proof of Theorem \ref{mainthm2}] 
%The construction presented in the previous section is applicable in this scenario. 
Let us briefly recall the proof of the asymptotic \eqref{eq:asymptotic}.
Start by noting that two linearly independent solutions of the associated homogeneous equation
$w_n''+4\pi^2\lambda_n w_n=0$
are given by
\begin{equation*}%\label{w1w2}
w_{n,1}(x):=\cos(2\pi\sqrt{\lambda_n}x),\text{ and }w_{n,2}(x):=\sin(2\pi\sqrt{\lambda_n}x),
\end{equation*}
and have constant Wronskian 
\begin{displaymath}
W(w_n^{(1)},w_n^{(2)}):=
\det\left(
\begin{array}{cc}
w_{n,1} & w_{n,2}\\
w'_{n,1} & w'_{n,2}
\end{array}\right)
=2\pi\sqrt{\lambda_n}.
\end{displaymath}
The general solution to the eigenvalue problem \eqref{eq:EgvProb} is then given by%\footnote{F,S: This was wrong. Please check the new formula \eqref{PreGenSolNHODE2}.}
\begin{equation}\label{PreGenSolNHODE2}
a \cos(2\pi \sqrt{\lambda_n}x)+ b \sin(2\pi \sqrt{\lambda_n}x)+ \frac{2\pi}{\sqrt{\lambda_n}}\int_{0}^{x} \sin(2\pi\sqrt{\lambda_n}(x-t))V(t) w_n(t)\,\d t,
\end{equation}
for some $a,b\in\R$, as can be easily checked by direct differentiation.
Evaluating \eqref{PreGenSolNHODE2} and its derivative at zero while appealing to hypotheses (H1) and (H3), we then have that
\begin{equation}\label{GenSolNHODE2}
w_n(x)=\sqrt{w_n(0)^2+\tfrac{w_n'(0)^2}{4\pi^2\la_n}} \, 
\cos\left(2\pi\left(\sqrt{\lambda_{n}}x-\tfrac{n}4\right)\right)
+\frac{2\pi}{\sqrt{\lambda_n}}\int_{0}^{x} \sin(2\pi\sqrt{\lambda_n}(x-t))V(t) w_n(t)\,\d t.
\end{equation}
%where $c_n = (w_n(0)^2+w_n'(0)^2/\la_n)^{1/2}$. 
Define $M_n(x):=\max\{|w_n(y)|: y\in[0,x]\}$. 
Applying the integral form of Gr\"onwall's inequality \cite[Theorem 1.10]{gron1} to \eqref{GenSolNHODE2}, we deduce 
$$M_n(x) \leq \left(w_n(0)^2+\tfrac{w_n'(0)^2}{4\pi^2\la_n}\right)^{1/2} + o_n(1)M_n(x),$$
and therefore  
$$M_n(x) \lesssim \left(w_n(0)^2+\tfrac{w_n'(0)^2}{4\pi^2\la_n}\right)^{1/2},$$ 
from where asymptotic \eqref{eq:asymptotic} follows at once.

The rest of the proof follows similar steps to those of Theorem \ref{mainthm}.
Firstly, we can restrict attention to the case of rational $x/y$. Secondly,
\begin{equation*}
\ell_{\{w_n\}}(x,y)=\ell_{\{v_n\}}(x,y),
\end{equation*}
where $v_n(x):=\cos(2\pi(\sqrt{\lambda_n} x-\frac n4))$.
Thirdly, given the equidistribution assumption (H2), Lemma \ref{limit-lemma} again applies and reduces the computation to 
$$\ell_{\{v_n\}}(x,y)=\int_0^1 \left(\frac{\Psi_0+\Psi_1}2\right)(pt, qt)\, \d t.$$
Here, the functions $\Psi_0,\Psi_1$ are given by $\Phi_0=:2\Psi_0-1$ and $\Phi_1=:2\Psi_0-1$, where $\Phi_0:=\Phi$ was given in \eqref{eq:DefPhi}, and $\Phi_1(x,y):=\Phi(x-\frac14,y-\frac14)$.
Consequently,
$$\ell_{\{w_n\}}(x,y)=\frac12+\frac14\int_0^1\Phi(pt, qt)\,\d t+\frac14\int_0^1\Phi(pt-\tfrac14, qt-\tfrac14)\,\d t.$$
These integrals can be calculated with Lemma \ref{lem:Fourier}. Invoking it with $\alpha=\beta=0$, and then with $\alpha=\beta=\frac14$, yields
$$\ell_{\{w_n\}}(x,y)=\frac12+\frac1{4{pq}}\left({(-1)^{\frac{p+q}2+1}}+(-1)^{p+1}\right)$$
This is readily seen to be equivalent to the result as stated in \eqref{eq:SCLpotentials}.
\end{proof}

\section{Further examples: Hermite functions, Laguerre polynomials and sets of bounded remainder}\label{examples}

\subsection{Hermite functions}
One could think of replacing hypothesis \eqref{eq:H1asympt} from Theorem \ref{mainthm} by a less restrictive assumption of the form
$$w_n(x) = \left(1 + o_n(1) \right) \phi(x,n) \cos{\left( 2\pi (\mu_n \varphi(x) - \theta_n) \right)},$$
where $\{\theta_n\}$ is now a {\it sequence}.
Without any further assumption, several steps of the preceding proofs break down completely. 
However, if some quantitative control on the speed with which the sequences $\{\mu_n \varphi(x)\}$ and 
$\{\mu_n \varphi(y)\}$ equidistribute modulo $1$ is known, then we can allow for a certain degree of variability in the sequence $\{\theta_n\}$.
It is not clear to us what the sharp version of such a statement would be, and we leave it for future research.\\

Cases in which the sequence $\{\theta_n\}$ changes rapidly with $n$, but does so in a structured manner, are also of interest. 
Such cases may be dealt with by partitioning $\{w_n\}$ into an appropriate number of subsequences, as we now illustrate.
A particularly nice example which fits into this framework (and served as original inspiration for Theorem \ref{mainthm2}) is that of the Schr\"odinger operator on the real line,
$$H:=-\frac{1}{4\pi^2}\frac{\d^2}{\d x^2}+x^2.$$
The operator $H$ is diagonalized by the Hermite functions,
$$
\varphi_n(x) := H_n(\sqrt{2\pi}x)e^{-\pi x^2}.
$$
Here, $\{H_n(x)\}$ denote the classical Hermite polynomials, which are orthogonal with respect to the standard Gaussian measure $e^{-\pi x^2}\, \d x$. As is well-known,
\begin{equation}\label{eq:EigenValueVarPhi}
H (\varphi_n) = \frac{2n+1}{2\pi}\varphi_n.
\end{equation}
Moreover, the  asymptotic from \cite[Theorem 8.22.6 and Formula (8.22.8)]{Sz75},
$$ 
 H_n(\sqrt{2\pi}x) e^{-{\pi x^2}}
= (1+o_n(1))\frac{\Gamma(n+1)}{\Gamma(\frac n2+1)}\cos{  \left(2\pi \left(\sqrt{\tfrac{2n+1}{2\pi}}x - \frac{n}{4}\right) \right)}
$$
shows that the eigenfunctions $\{\varphi_n\}$ in \eqref{eq:EigenValueVarPhi} do {not} satisfy the assumptions of Theorem 1,
% (the condition on $\theta$ being constant is violated)
 but that the subsequences $\{\varphi_{2n}\}$ and $\{\varphi_{2n+1}\}$ do. We further note that the basis $\{\varphi_n\}$  diagonalizes the Fourier transform in the following sense: the elements of $\{\varphi_n\}$  are pairwise orthogonal, dense in $L^2(\R)$, and 
$$
\widehat \varphi_n(\xi) = \int_{-\infty}^\infty \varphi_n(x)e^{-2\pi i \xi x}\,\d x = (-i)^n \varphi_n(\xi).
$$ 

A simple consequence of Theorem \ref{mainthm2} (plus a short computation) is the following.

\begin{proposition}[Sign Correlations for Hermite functions] Let $x,y \neq 0$ and $y/x \in \mathbb{Z}$. Then
 $$  \ell_{\{\varphi_n\}}(x,y) = \begin{cases}
\frac{1}{2} + \frac{x}{2y}~ &\mbox{if}~\frac{y}{x} \equiv 1 \mod 4, \\
 \frac{1}{2} \qquad &\mbox{otherwise.}
 \end{cases}$$
 \end{proposition}

\subsection{Laguerre polynomials} 
An extension of the previous example to higher dimensions involves the so-called {\it Laguerre polynomials}. Let $\{L^\nu_n(x)\}$ be the (generalized) Laguerre polynomials with parameter $\nu>-1$, defined via
\begin{equation}\label{Laguerre_orthogonality}
\int_{0}^\infty L_n^\nu(x)L_m^\nu(x)x^{\nu}e^{-x}\, \d x = \frac{\Gamma(n+\nu+1)}{n!}\delta(n-m).
\end{equation}
 It is well-known, see \cite[Formula (8.22.2)]{Sz75}, that
$$
L_n^\nu(2\pi x^2) e^{-\pi x^2}=(1+o_n(1))\frac{{n}^{\nu/2-1/4}}{\sqrt{\pi}(2\pi)^{\nu/2+1/4}x^{\nu+1/2}} \cos\left(2\pi\left(\sqrt{\tfrac{4n+2\nu+2}{2\pi}}x - \frac{2\nu+1}{8}\right)\right).
$$
It is also known that the set of Laguerre functions 
$$
{\bf x}\in \R^d \mapsto \Phi_n({\bf x}):=L^{\nu}_n(2\pi|{\bf x}|^2) e^{-\pi|{\bf x}|^2},
$$
with $\nu=d/2-1$,  diagonalizes the operator ${\bf H}=-\frac{1}{4\pi^2}\Delta + |{\bf x}|^2$ over the space of radial functions in $\R^d$, and that
$$
{\bf H}(\Phi_n) = \frac{(4n+2\nu+2)}{2\pi} \Phi_n.
$$
We also note that $\{\Phi_n\}$  diagonalizes the Fourier transform over the space of square integrable radial functions in $\R^d$. Indeed,
$$
\widehat \Phi_n(\boldsymbol{\xi}) = \int_{\R^d} \Phi_n({\bf x})e^{-2\pi i \boldsymbol{\xi} \cdot {\bf x}}\,\d {\bf x} =  (-1)^n \Phi_n(\boldsymbol{\xi}).
$$
The following result is a direct application of Lemma \ref{lem:Fourier} with $\alpha=\beta=\frac{2\nu+1}{8} = \frac{d-1}{8}$.
\begin{proposition}[Sign Correlations for Laguerre functions]
Let $r_1,r_2>0$ be radii such that $\frac{r_1}{r_2} = \frac{p}{q}$ for some coprime integers $p$ and $q$. Then
\begin{equation*}
\ell_{\{\Phi_n\}}(r_1,r_2) = \left\{
\begin{array}{lc}
  \frac{1}{2}, \ {\rm if}\  p\text{ or } q \text{ is even}, \ \text{ or if} \ p,q \text{ and }  \frac{(p-q)(d-1)}2 \text{ are odd}, \medskip \\
 \frac{1}{2} - \frac{1}{2pq} (-1)^{\frac{p+q}2 + \frac{(p-q)(d-1)}4}, \ \text{otherwise}.
\end{array}
\right.
\end{equation*}
\end{proposition}

\subsection{Sets of bounded remainder}
In this final section, we describe a  curious phenomenon which was discovered by accident. Consider the family of Chebyshev polynomials of the first kind on the interval $[-1,1]$, denoted $\{T_n\}_{n\geq 0}$ and defined via $T_n(x):=\cos(n\arccos x)$. 
%We illustrate our main result by studying the quantity
%$$ \# \left\{0 \leq k < n: \sgn T_k(\tfrac{1}{10}) = \sgn T_k(\tfrac{3}{10})\right\}.$$
This turns out to be one of the extremal examples for  Theorem \ref{mainthm} since
$$ \# \left\{0 \leq n < N: \sgn T_n(\cos{(\tfrac{2\pi}{10})}) = \sgn T_n(\cos{(3\tfrac{2\pi}{10})})\right\} = (1+o_N(1)) \tfrac{N}{3}.$$

Indeed, the arising quantities simplify to 
$$ T_n(\cos{(\tfrac{2\pi}{10})}) = \cos{(\tfrac{2\pi n}{10})} \qquad \mbox{and} \qquad  T_n(\cos{(3\tfrac{2\pi}{10})}) = \cos{(3\tfrac{ 2\pi n}{10})},$$
and both sequences $\{\frac{2\pi n}{10}\}$ and $\{3\frac{ 2\pi n}{10}\}$ are equidistributed modulo 1. 
If we go one step further and try to understand the error term, we encounter the following surprising phenomenon. At least for $N \leq 10^4$, we used  {\it Mathematica} to verify that
$$ \left| \# \left\{0 \leq n < N: \sgn T_n(\cos{(\tfrac{2\pi}{10})}) = \sgn T_n(\cos{(3\tfrac{ 2\pi}{10})})\right\} - \tfrac{N}{3} \right| \leq 10.$$
This is probably related to the fine structure of Kronecker sequences \cite{hecke, kesten, sos}, and
one cannot hope for such strong results in general.
This is reminiscent of exciting new developments in the theory of continuous flows on the torus,  Jozs\'{e}f Beck's superuniformity theory \cite{beck0, beckh, beck, beck2}, that may have nontrivial implications to the present context. We also refer to a related paper by Grepstad \& Larcher on sets of bounded remainder \cite{grep}, which seems to provide further interesting directions of research.

%----------------------REFERENCES-------------------------------

%----------------------THE--END!-----------------------------------


\begin{thebibliography}{99}

\bibitem{beck0} \textsc{J. Beck},
\newblock Strong uniformity. Uniform distribution and quasi-Monte Carlo methods, 17--44, 
\newblock Radon Ser. Comput. Appl. Math. {\bf 15}, De Gruyter, Berlin, 2014. 

\bibitem{beckh} \textsc{J. Beck}, 
\newblock Super-uniformity of the typical billiard path. An irregular mind, 39--129, 
\newblock Bolyai Soc. Math. Stud. {\bf 21}, J\'anos Bolyai Math. Soc., Budapest, 2010. 

\bibitem{beck} \textsc{J. Beck}, 
\newblock From Khinchin's conjecture on strong uniformity to superuniform motions.
\newblock Mathematika {\bf 61} (2015), no.~3, 591--707. 

\bibitem{beck2} \textsc{J. Beck},
\newblock Dimension-free uniformity with applications, I. 
\newblock Mathematika {\bf 63} (2017), no.~3, 734--761. 

\bibitem{dick} \textsc{J. Dick and F. Pillichshammer}, 
\newblock Digital nets and sequences. Discrepancy theory and quasi-Monte Carlo integration. 
\newblock Cambridge University Press, Cambridge, 2010.

\bibitem{drmota} \textsc{M. Drmota and R. Tichy}, 
\newblock Sequences, discrepancies and applications. 
\newblock Lecture Notes in Mathematics, 1651. Springer-Verlag, Berlin, 1997.

\bibitem{fefferman} \textsc{C. Fefferman}, 
\newblock The uncertainty principle. 
\newblock Bull. Amer. Math. Soc. {\bf 9} (1983), no.~2, 129--206.

\bibitem{us3} \textsc{F. Gon\c{c}alves, D. Oliveira e Silva and S. Steinerberger}, 
\newblock Hermite polynomials, linear flows on the torus, and an uncertainty principle for roots.
\newblock J. Math. Anal. Appl. {\bf 451} (2017), no. 2, 678--711.

\bibitem{grep} \textsc{S. Grepstad and G. Larcher}, 
\newblock Sets of bounded remainder for a continuous irrational rotation on $[0,1]^2$. 
\newblock Acta Arith. {\bf 176} (2016), no.~4, 365--395. 

\bibitem{hecke} \textsc{E. Hecke}, 
\newblock \"Uber analytische Funktionen und die Verteilung von Zahlen mod. eins.  
\newblock Abh. Math. Sem. Univ. Hamburg {\bf 1} (1922), no.~1, 54--76. 


\bibitem{kesten} \textsc{H. Kesten}, 
\newblock On a conjecture of Erd\"os and Sz\"usz related to uniform distribution mod 1. 
\newblock Acta Arith. {\bf 12} 1966/1967 193--212. 

\bibitem{kuipers} \textsc{L. Kuipers and H. Niederreiter}, 
\newblock Uniform distribution of sequences.
\newblock Pure and Applied Mathematics. Wiley-Interscience, New York-London-Sydney, 1974.

\bibitem{RS72} 
\textsc{M. Reed and B. Simon}, 
\newblock Methods of modern mathematical physics. I. Functional analysis. 
\newblock Academic Press, New York-London, 1972.

\bibitem{sos} \textsc{V. T. S\'os}, 
\newblock On the distribution mod 1 of the sequence $n \alpha$, 
\newblock Ann. Univ. Sci. Budapest, E\"otv\"os Sect. Math. {\bf 1} (1958), 127--134.


\bibitem{stein} \textsc{S. Steinerberger}, 
\newblock Maximizing smooth functions over toroidal geodesics.
\newblock arXiv:1805.02601. To appear in Bull. Aust. Math. Soc.

\bibitem{Sz75} \textsc{G. Szeg\"o}, 
\newblock Orthogonal polynomials. 
Fourth edition. American Mathematical Society, Colloquium Publications, Vol. XXIII. 
\newblock American Mathematical Society, Providence, R.I., 1975. 



\bibitem{gron1} \textsc{T. Tao}, 
\newblock
Nonlinear dispersive equations. Local and global analysis. 
\newblock CBMS Regional Conference Series in Mathematics, 106. Published for the Conference Board of the Mathematical Sciences, Washington, DC; by the American Mathematical Society, Providence, RI, 2006.


\end{thebibliography}
\end{document}